\documentclass[12pt]{amsart}

\input{prepictex}
\input{pictex}
\input{postpictex}

\newtheorem{thm}{Theorem}[section]

\newtheorem{lem}[thm]{Lemma}

\newtheorem{cor}[thm]{Corollary}

\newtheorem{prop}[thm]{Proposition}

\theoremstyle{definition}

\newtheorem{defn}[thm]{Definition}

\newtheorem{remark}[thm]{Remark}

\newcommand{\NN}{{\Bbb N}}

\newcommand{\QQ}{{\Bbb Q}}

\newcommand{\ZZ}{{\Bbb Z}}

\newcommand{\cE}{{\mathcal E}}
\newcommand{\cF}{{\mathcal F}}

\newcommand{\cO}{{\mathcal O}}

\newcommand{\cT}{{\mathcal T}}

\newcommand{\G}{{\Gamma}}
\newcommand{\Om}{{\Omega}}
\newcommand{\s}{{\sigma}}
\newcommand{\w}{{\omega}}

\newsymbol\rtimes 226F 

\newcommand{\factor}{{L^{\infty}(\Om)\rtimes\G}}
\newcommand{\aut}{{\text{\rm Aut}}}
\newcommand{\tl}{{{\text{\rm{III}}}_{\lambda}}}
\newcommand{\tn}{{{\text{\rm{III}}}_{1/n}}}
\newcommand{\tone}{{{\text{\rm{III}}}_{1}}}
\newcommand{\tp}{{{\text{\rm{III}}}_{1/p}}}

\begin{document}

\title{Factors from trees}
\author{Jacqui Ramagge}
\address{Mathematics Department, University of Newcastle, Callaghan, NSW 2308,
Australia}
\email{jacqui@maths.newcastle.edu.au}
\author{Guyan Robertson}
\address{Mathematics Department, University of Newcastle, Callaghan, NSW 2308,
Australia}
\email{guyan@maths.newcastle.edu.au}

\maketitle

\begin{abstract}
We construct factors of type $\tn$  for $n\in\NN, n\geq 2$ from group actions on
homogeneous trees and their boundaries.  Our result is a discrete analogue of a
result of R.J~Spatzier, namely~\cite[Proposition~1]{Spa},  where the
hyperfinite factor of type $\tone$ is constructed from a group action on the
boundary of the universal cover of a manifold.
\end{abstract}

\section{Introduction}

Let $\G$ be a group acting simply transitively on the vertices of a homogeneous
tree $\cT$ of degree $n+1<\infty$.  Then, by~\cite[Ch.~I, Theorem~6.3]{FTN},
\[
\G\cong\ZZ_2 * \cdots * \ZZ_2* \ZZ * \cdots * \ZZ
\]
where there are $s$ factors of $\ZZ_2$, $t$ factors of $\ZZ$, and $s+2t=n+1$.
Thus $\G$ has a presentation
\[
\G = \left< a_1 ,\ldots ,a_{s+t} : a_{i}^2 = 1 \text{ for } i\in \{ 1,
\ldots , s\} \right>,
\]
we can identify the Cayley graph of $\G$ constructed via right
multiplication with $\cT$ and the action of $\G$ on $\cT$ is equivalent
to the natural action of $\G$ on its Cayley graph via left multiplication.

We can associate a natural boundary to $\cT$, namely the set $\Om$ of
semi-infinite reduced words in the generators of $\G$. The action of $\G$
on $\cT$
induces an action of $\G$ on $\Om$.

For each $x\in\G$, let
\[
\Om^x = \{ \w\in\Om : \w= x\cdots \}
\]
be the set of semi-infinite reduced words beginning with $x$. The set
$\left\{ \Om^x \right\}_{x\in\G}$ is a set of basic open sets for a compact
Hasudorff topology on $\Om$. Denote by $|x|$ the length of a reduced
expression for $x$. Let $V^m=\{x\in\G : |x|=m\}$ and define
$N_m=\left| V^m\right|$. Then $\Om$ is the disjoint union of the $N_m$ sets
$\Om^x$ for $x\in V^m$.

We can also endow $\Om$ with the structure of a measure space. $\Om$ has a
unique distinguished Borel probability measure $\nu$ such that
\[
\nu\left(\Om^x\right) = \frac{1}{n+1}\left(\frac{1}{n}\right)^{|x|-1}
\]
for every nontrivial $x\in\G$. The sets $\Om^x$, $x\in\G$ generate the Borel
$\s$-algebra.

This measure $\nu$ on $\Om$ is quasi-invariant under the action of $\G$, so that
$\G$ acts on the measure space $(\Om,\nu)$ and enabling us to extend the
action of
$\G$  to an action  on $L^{\infty}(\Om,\nu)$ via
\[
g\cdot f(\w) = f (g^{-1}\cdot\w)
\]
for all $g\in\G$, $f\in L^{\infty}(\Om,\nu)$, and $\w\in\Om$. We may therefore
consider the von Neumann algebra $L^{\infty}(\Om,\nu)\rtimes\G$ which we shall
write as  $\factor$ for brevity.

\section{The Factors}

We note that the action of $\G$ on $\Om$ is free since if $g\w=\w$ for some
$g\in\G$ and $\w\in\Om$ then we must have either $\w=ggg\cdots$ or
$\w=g^{-1}g^{-1}g^{-1}\cdots$ and
\[
\nu\left( { ggg\cdots, g^{-1}g^{-1}g^{-1}\cdots}\right) =0.
\]

The action of $\G$ on $\Om$ is also ergodic by the proof
of~\cite[Proposition~3.9]{PS}, so that $\factor$ is a factor. Establishing
the type of the factor is not quite as straightforward. We begin by
recalling some
classical definitions.

\begin{defn}
Given a group $\G$ acting on a measure space $\Om$, we define the {\bf full
group}, $[\G]$, of $\G$ by
\[
[\G]=\left\{ T\in\aut(\Om) : T \w\in \G\w \text{ for almost every
}\w\in\Om\right\}.
\]
The set $[\G]_0$ of measure preserving maps in $[\G]$ is then given by
\[
[\G]_0 =\left\{ T\in [\G] : T{\scriptstyle\circ}\nu =\nu\right\}
\]
\end{defn}

\begin{defn}
Let $G$ be a countable group of automorphisms of the measure space $(\Om,\nu)$.
Following W.~Krieger, define the {\bf ratio set} $r(G)$ to be the subset of
$[0,\infty)$ such that if $\lambda\geq 0$ then $\lambda\in r(G)$ if and
only if for
every $\epsilon>0$ and Borel set $\cE$ with $\nu(\cE)>0$, there exists a
$g\in G$
and a Borel set $\cF$ such that $\nu(\cF)>0$, $\cF\cup g\cF\subseteq\cE$ and
\[
\left| \frac{d\nu{\scriptstyle\circ}g}{d\nu}(\w) -\lambda \right| <\epsilon
\]
for all $\w\in\cF$.
\end{defn}

\begin{remark}\label{full group determines ratio set}
The ratio set $r(G)$ depends only on the quasi-equivalence class of the measure
$\nu$, see~\cite[\S I-3, Lemma 14]{HO}. It also depends only on the full
group in the
sense that
\[
[H]=[G] \Rightarrow r(H)=r(G).
\]
\end{remark}

The following result will be applied in the special case where
$G=\G$. However, since the simple transitivity of the action doesn't play a
role in
the proof, we can state it in greater generality.

\begin{prop}\label{ratio set of G}
Let $G$ be a countable subgroup of $\aut(\cT)\leq\aut(\Om)$. Suppose there exist
an element $g\in G$ such that $d(ge,e)=1$ and a subgroup $K$ of $[G]_0$ whose
action on $\Om$ is ergodic. Then
\[
r(G)=\left\{ n^{k}: k\in\ZZ\}\cup\{ 0\right\}.
\]
\end{prop}
\begin{proof}
By Remark~\ref{full group determines ratio set}, it is
sufficient to prove the statement for some group $H$ such that $[H]=[G]$. In
particular, since $[G]= [\left< G, K\right> ]$ for any subgroup $K$ of
$[G]_0$, we may
assume without loss of generality that $K\leq G$.

By~\cite[Chapter~2, part 1)]{FTN}, for each $g\in G$ and $\w\in\Om$ we have
\[
\frac{d\nu{\scriptstyle\circ}g}{d\nu}(\w)\in\left\{ n^{k}: k\in\ZZ\}\cup\{
0\right\}.
\]
Since $G$ acts ergodically on $\Om$, $r(G)\setminus\{ 0\}$ is a group. It is
therefore enough to show that $n\in r(G)$.  Write $x=ge$ and note that
$\nu_x=\nu{\scriptstyle\circ}g^{-1}$. By ~\cite[Chapter~2, part1)]{FTN} we
have
\begin{equation}\label{q2 derivative}
\frac{d\nu_x}{d\nu}(\w) =n, \text{ for all } \w\in\Om_e^x.
\end{equation}
Let $\cE\subseteq\Om$ be a Borel set with $\nu(\cE)>0$. By the ergodicity of $K$,
there exist $k_1,k_2\in K$ such that the set
\[
\cF=\{\w\in \cE: k_1\w\in\Om_e^x \text{ and } k_2g^{-1}k_1\w\in \cE\}
\]
has positive measure.

Finally, let $t=k_2g^{-1}k_1\in G$. By construction, $\cF\cup t\cF\subseteq
\cE$.
Moreover, since $K$ is measure-preserving,
\[
\frac{d\nu{\scriptstyle\circ}t}{d\nu}(\w)=
\frac{d\nu{\scriptstyle\circ}g^{-1}}{d\nu}(k_1\w)=
\frac{d\nu_x}{d\nu}(k_1\w) =n \text{ for all } \w\in \cF
\]
by~\eqref{q2 derivative}, since $k_1\in\Om_e^x$. This proves $n\in r(G)$, as
required.
\end{proof}

\begin{cor}\label{factor if free and ergodic}
If, in addition to the hypotheses for Proposition~\ref{ratio set of G}, the
action of
$G$ is free, then $L^\infty(\Om)\rtimes G$ is a factor of type $\tn$.
\end{cor}
\begin{proof}
Having determined the ratio set, this is immediate
from~\cite[Corollaire~3.3.4]{Connes73}.
\end{proof}

Thus, if we can find a countable subgroup $K\leq [\G]_0$ whose action on
$\Om$ is
ergodic we will have shown that $\factor$ is a factor of type $\tn$. To
this end, we
prove the following sufficiency condition for ergodicity.

\begin{lem}\label{l:ergodicity}
Let $K$ be group which acts on $\Om$. If $K$ acts transitively on the
collection of
sets $\left\{ \Om^x :  x\in\G , |x|=m \right\}$ for each natural number
$m$, then $K$
acts ergodically on $\Om$.
\end{lem}
\begin{proof}
Suppose that $X_0\subseteq\Om$ is a Borel set which is invariant under $K$ and
such that $\nu(X_0)>0$. We show that this necessarily implies
$\nu(\Om\setminus X_0)=0$, thus establishing the ergodicity of the action.

Define a new measure $\mu$ on $\Om$ by $\mu(X)=\nu(X\cap X_0)$ for each Borel
set $X\subseteq\Om$. Now, for each $g\in K$,
\begin{eqnarray*}
\mu(gX)&=&\nu(gX\cap X_0) = \nu(X\cap g^{-1}X_0) \\
&\leq & \nu(X\cap X_0) + \nu(X\cap (g^{-1}X_0\setminus X_0)) \\
&=& \nu(X\cap X_0) \\
&=& \mu(X),
\end{eqnarray*}
and therefore $\mu$ is $K$-invariant. Since $K$ acts transitively on the
basic open
sets $\Om^x$ associated to words
$x$ of length $m$ this implies that
\[
\mu(\Om^x)=\mu(\Om^y)
\]
whenever $|x|=|y|$. Since $\Om$ is the union of $N_m$ disjoint sets
$\Om^x$, $x\in
V^m$, each of which has equal measure with respect to $\mu$, we deduce that
\[
\mu(\Om^x)=\frac{c}{N_m}
\]
for each $x\in V^m$, where $c=\mu(X_0)=\nu(X_0)>0$. Thus
$\mu(\Om^x)=c\nu(\Om^x)$ for every $x\in\G$.

Since the sets $\Om^x$, $x\in\G$ generate the Borel $\s$-algebra, we deduce that
$\mu(X)=c\nu(X)$ for each Borel set $X$. Therefore
\begin{eqnarray*}
\nu(\Om\setminus X_0) &=& c^{-1}\mu(\Om\setminus X_0) \\
&=& c^{-1}\nu((\Om\setminus X_0) \cap X_0) = 0 ,
\end{eqnarray*}
thus proving ergodicity.
\end{proof}

In the last of our technical results, we give a constructive proof of the
existence
of a countable ergodic subgroup of $[\G]_0$.

\begin{lem}\label{ergodic subgroup of full group}
There is a countable ergodic group $K\leq\aut(\Om)$ such that $K\leq [\G]_0$.
\end{lem}
\begin{proof}
Let $x,y\in V^m$. We construct a measure preserving automorphism $k_{x,y}$ of
$\Om$ such that
\begin{enumerate}
\item $k_{x,y}$ is almost everywhere a bijection from $\Om^x$ onto $\Om^y$,
\item $k_{x,y}$ is the identity on $\Om\setminus(\Om^x\cup\Om^y)$.
\end{enumerate}
It then follows from Lemma~\ref{l:ergodicity} that the group
\[
K=\left< k_{x,y} : \{x,y\}\subseteq V^m, m\in\NN \right>
\]
acts ergodically on $\Om$ and the construction will show explicitly that
$K\leq [\G]_0$.

Fix $x,y\in V^m$ and suppose that we have reduced expressions $x=x_1\ldots
x_m$, and $y=y_1\ldots y_m$.

Define $k_{x,y}$ to be left multiplication by $yx^{-1}$ on each of the sets
$\Omega^{xz}$ where $|z| = 1$ and $z \notin \{x_m^{-1},y_m^{-1}\}$.
Then $k_{x,y}$ is a measure preserving bijection from each such set onto
$\Omega^{yz}$. If $y_m = x_m$ then $k_{x,y}$ is now well defined everywhere on
$\Omega^x$.

Suppose now that $y_m \ne x_m$. Then $k_{x,y}$ is defined on the set $\Omega^x
\setminus \Omega^{xy_m^{-1}}$ ,  which it maps bijectively onto $\Omega^y
\setminus \Omega^{yx_m^{-1}}$. Now define $k_{x,y}$ to be left multiplication by
$yx_m^{-1}y_mx^{-1}$ on each of the sets $\Omega^{xy_m^{-1}z}$ where $|z| = 1$
and  $z \notin \{x_m, y_m \}$.  Then $k_{x,y}$ is a measure preserving
bijection of
each such $\Omega^{xy_m^{-1}z}$ onto $\Omega^{yx_m^{-1}z}$.

Thus we have extended the domain of $k_{x,y}$ so that it is now defined on
the set
$\Omega^x \setminus \Omega^{xy_m^{-1}x_m}$, which it maps bijectively onto
$\Omega^y \setminus \Omega^{yx_m^{-1}y_m}$.

Next define $k_{x,y}$ to be left muliplication by
$yx_m^{-1}y_mx_m^{-1}y_mx^{-1}$ on the sets $\Omega^{xy_m^{-1}x_mz}$  where
$|z| = 1$ and $z \notin  \{x_m^{-1},y_m^{-1}\}$.

Continue in this way. At the jth step $k_{x,y}$ is a measure preserving
bijection from $\Omega^x \setminus X_j$ onto $\Omega^y \setminus Y_j$
where $\nu(X_j)\rightarrow 0$  as $j\rightarrow \infty$ so that eventually
$k_{x,y}$ is defined almost everywhere on $\Om$. Finally, define
\[
k_{x,y}(xy_m^{-1}x_my_m^{-1}x_my_m^{-1}x_m\ldots) =
yx_m^{-1}y_mx_m^{-1}y_mx_m^{-1}y_m\ldots
\]
thus defining $k_{x,y}$ everywhere on $\Om$ in such a way that its action is
pointwise approximable by $\G$ almost everywhere. Hence
\[
K=\left< k_{x,y} : \{x,y\}\subseteq V^m, m\in\NN \right>
\]
is a countable group with an ergodic measure-preserving  action on $\Om$ and
$K\leq [\G]_0$.
\end{proof}

We are now in a position to prove our main result.

\pagebreak[2]
\begin{thm}
The von Neumann algebra $\factor$ is the hyperfinite factor of
type~$\tn$.
\end{thm}
\begin{proof}
By applying Corollary~\ref{factor if free and ergodic} with $G=\G$, $g\in\G$ any
generator of $\G$, and $K$ as in Lemma~\ref{ergodic subgroup of full group} we
conclude that $\factor$ is a factor of type $\tn$.

To see that the factor is hyperfinite simply note that the action of $\G$ is
amenable as a result of~\cite[Theorem~5.1] {Adams}. We refer
to~\cite[Theorem~4.4.1]{Connes78} for the uniqueness of the hyperfinite
factor of
type $\tn$.
\end{proof}

\begin{remark}
Taking different measures on $\Om$ should yield hyperfinite factors of type
$\tl$
for any $0<\lambda<1$. We have concentrated on the geometrically interesting
case.
\end{remark}

\begin{remark}
In~\cite{Spielberg1990}, Spielberg constructs $\tl$  factor states on the
algebra
$\cO_2$. The reduced $C^*$-algebra $C(\Om)\rtimes_r\G$ is a Cuntz-Krieger
algebra $\cO_A$ by~\cite{Spielberg1991}.  What we have done is construct a type
$\tn$   factor state on some of these algebras $\cO_A$.
\end{remark}

\begin{remark}
From~\cite[p.~476]{Connes78}, we know that if $\G=\QQ\rtimes\QQ^*$ acts
naturally on $\QQ_p$, then the crossed product $L^\infty(\QQ_p)\rtimes\G$ is the
hyperfinite factor of type $\tp$. This may be proved geometrically as above by
regarding the the boundary of the homogeneous tree of degree $p+1$ as the one
point compactification of $\QQ_p$ as in~\cite{CKW}.
\end{remark}

\end{document}